\documentclass[a4paper,10pt,oneside,reqno]{amsart}
\pdfoutput=1

\usepackage[activate={true,nocompatibility},final,tracking=false,kerning=true,spacing=true,factor=1100,stretch=0,shrink=0]{microtype}

\usepackage[T1]{fontenc} 

\usepackage[utf8]{inputenc}
\usepackage[british]{babel}
\usepackage{mathrsfs}
\usepackage{amsfonts}
\usepackage{amsmath} 
\usepackage{amsthm}  
\usepackage{amssymb}
\usepackage{mathtools}
\mathtoolsset{showonlyrefs}
\usepackage{graphicx}
\usepackage{xcolor}
\usepackage{url}
\usepackage{cite}
\usepackage{xifthen}
\usepackage{stmaryrd}
\usepackage{centernot}
\usepackage{indentfirst}
\usepackage[hang,flushmargin]{footmisc}
\usepackage{geometry}
\usepackage{hyperref}
\hypersetup{colorlinks=false, pdfborder={0 0 0}}
\usepackage{cleveref}
\usepackage{tikz}
\usetikzlibrary{matrix,arrows,decorations.pathmorphing}

\usepackage{enumitem}
\setlist[1]{labelindent=\parindent}
\setlist[enumerate,1]{label = \arabic*.,ref   = \arabic*}
\setlist[enumerate,2]{label = \alph*),ref   = \theenumi\alph*)}
\setlist[enumerate,3]{label = \roman*),ref   = \theenumii.\roman*}

\newcommand{\from}{\colon}
\renewcommand{\models}{\vDash}
\newcommand{\proves}{\vdash}
\newcommand{\monster}{\mathfrak U}
\newcommand{\eq}{\mathrm{eq}}
\newcommand{\meet}{\sqcap}

\DeclarePairedDelimiter{\set}{\{}{\}}

\theoremstyle{definition}
\newtheorem{defin}{Definition}[section]

\newtheorem{pr}[defin]{Proposition}
\newtheorem{co}[defin]{Corollary}
\newtheorem{lemma}[defin]{Lemma}
\newtheorem{question}[defin]{Question}

\title{Some definable types that cannot be amalgamated}

\author{Martin Hils}
\address{Martin Hils, Institut f\"ur Mathematische Logik und Grundlagenforschung, Westf\"alische Wilhelms-Universit\"at M\"unster, Einsteinstr. 62, D-48149 M\" unster, Germany}
\email{hils@uni-muenster.de}
\thanks{}
\author{Rosario Mennuni}
\address{Rosario Mennuni, Dipartimento di Matematica, Universit\`a di Pisa, Largo Bruno Pontecorvo 5, 56127 Pisa, Italy}
\email{R.Mennuni@posteo.net}
\thanks{MH was partially supported by the German Research Foundation (DFG) via HI 2004/1-1 (part of the French-German ANR-DFG project GeoMod) and under Germany's Excellence Strategy EXC 2044-390685587, `Mathematics M\"unster: Dynamics-Geometry-Structure'.\\RM is supported by the project PRIN 2017: ``Mathematical Logic: models, sets, computability'' Prot.~2017NWTM8RPRIN}


\begin{document}
\begin{abstract}
  We exhibit a theory where definable types lack the amalgamation property.
\end{abstract}
\maketitle

If $q_0(x,y)$ and $q_1(x,z)$ are types over a given set $A$, both extending the same type $p(x)\in S(A)$, it is an easy exercise to show that there is $r(x,y,z)\in S(A)$ extending $q_0(x,y)$ and $q_1(x,z)$ simultaneously. In other words, in every theory, types have the \emph{amalgamation property}. Suppose now that $p,q_0,q_1$ all belong to some special class $\mathcal K$ of types, and consider the following question: among the amalgams $r$ as above, is there always one which belongs to $\mathcal K$? In this short note, we prove that, for $\mathcal K$ the class of definable types, the answer is in general negative.

By a fundamental result of Shelah, a complete theory is stable if and only if all types over  models are definable. Definable types, and the tightly related notion of \emph{stable embeddedness}, recently attracted considerable attention in unstable contexts as well.
For instance, Hrushovski isolated a criterion for elimination of imaginaries in terms of  density of definable types, which yielded a simplified proof of the classification of imaginaries in algebraically closed valued fields~\cite{Joh20}, and similar classifcation results in other (enriched)  henselian valued fields~\cite{RidVDF,HKR2016,HiRi}. In o-minimal theories,  stable embeddedness of elementary substructures corresponds to relative Dedekind completeness~\cite{MS}, and in benign theories of henselian valued fields stable embeddedness obeys an Ax--Kochen--Ershov principle~\cite{CD,touchard}.  Definable types are also central in Hrushovski and Loeser's celebrated work on  Berkovich analytifications~\cite{HL}: their  \emph{stable completions} of algebraic varieties are certain spaces of definable types which, crucially, form \emph{strict pro-definable sets}.

This brings us to the main motivation for the present paper. If $T$ is stable, then definable types over $M$ may be seen as a pro-definable set in $M^\eq$, albeit  this pro-definability need not be strict:  it follows from work of Poizat on \emph{belles paires}~\cite{poizat}  that, if $T$ is stable, then definable types over models 
form strict pro-definable sets if and only if $T$ is $\mathsf{nfcp}$, if and only if all belles paires of models of $T$ are $\omega$-saturated. 
In order to establish strict pro-definability for other spaces of definable types, Cubides, Ye and the first author~\cite{bp} recently introduced \emph{beautiful pairs} in an arbitrary $L$-theory $T$. 
Poizat's belles paires are beautiful, and his theory generalises smoothly to unstable $T$, provided the latter satisfies certain assumptions: an extension property called (EP), and the amalgamation property (AP) for definable types.

In stable theories, (AP) and (EP) always hold. Similarly in o-minimal theories, where (AP) follows from a result of Baisalov--Poizat~\cite{BaPo98} (see~\cite[Section~3.2]{bp}). In benign valued fields, there is an Ax--Kochen--Ershov type reduction for (AP)~\cite[Section~8]{bp}. In~\cite[Corollary~2.4.16]{bp} an  example of a (dp-minimal) theory satisfying (AP) but not (EP) was found, namely the levelled binary tree with level set $\omega$. Whether there is a theory where (AP) fails is left open in \cite{bp}.

Building on the aforementioned tree, we construct such an example.

\section{The theory}
Models $M$ of the theory in which our counterexample lives are four-sorted, and are roughly obtained as follows. We start with a binary tree $\operatorname{T}(M)$, with discrete level set $\operatorname{L}(M)$. We then introduce two levelled sets $\operatorname{A}(M)$ and $\operatorname{B}(M)$, both with the same level set as the tree, namely $\operatorname{L}(M)$, and cover each level $x$ of $\operatorname{T}(M)$  with a generic surjection from the cartesian product of the $x$-th levels of $\operatorname{A}(M)$ and $\operatorname{B}(M)$. In this section, we spell out this construction in detail.
\begin{defin}
  Let $L$ be the following language.
  \begin{enumerate}
  \item $L$ has four sorts $\mathrm{A}$, $\mathrm{B}$, $\mathrm{T}$, $\mathrm{L}$.
  \item $\mathrm{T}$ has a binary relation $\le_\mathrm{T}$, a binary function $\meet$, a unary function $\operatorname{pred}_\mathrm{T}$,  constants $g_\mathrm{T},r$.
  \item $\mathrm{L}$ has a binary relation $\le_\mathrm{L}$,   a unary function $\operatorname{pred}_\mathrm{L}$ and constants $g_\mathrm{L}$, $0$.
  \item There are functions $\ell_\mathrm{T}\from \mathrm{T}\to \mathrm{L}$, $\ell_\mathrm{A}\from \mathrm{A}\to \mathrm{L}$, and $\ell_\mathrm{B}\from \mathrm{B}\to \mathrm{L}$.
  \item There is a function $f\from \mathrm{A}\times \mathrm{B}\to \mathrm{T}$.
  \end{enumerate}
\end{defin}
\begin{defin}
  Let $T$ be the $L$-theory expressing the following properties.
  \begin{enumerate}[label=(\roman*)]
  \item $0\neq g_\mathrm{L}$, and $(\mathrm{L}\setminus \set{g_\mathrm{L}},\le_\mathrm{L})$ is a discrete linear order with smallest element 0 and no largest element, with predecessor function $\operatorname{pred}_\mathrm{L}$, with the convention that $\operatorname{pred}_\mathrm{L}(0)=0$. The ``garbage'' point $g_\mathrm{L}$ is not $\le_\mathrm{L}$-related to anything, and $\operatorname{pred}_\mathrm{L}(g_L)=g_L$.
  \item $(\mathrm{T}\setminus\set{g_\mathrm{T}}, \le_\mathrm{T}, \meet)$ is a meet-tree with root $r$, binary ramification,\footnote{We fix binary ramification for simplicity, but this is not important: any fixed finite ramification will work.} and, for every fixed element, its  set of predecessors  is a discrete linear order, with predecessor function $\operatorname{pred}_\mathrm{T}$, with the convention that $\operatorname{pred}_\mathrm{L}(r)=r$. The ``garbage'' point $g_\mathrm{T}$ behaves similarly to the garbage point $g_\mathrm{L}$.
  \item $\ell_\mathrm{T}$ is a a surjective level function $T\setminus \set{g_\mathrm{T}}\to L\setminus \set{g_\mathrm{L}}$, extended by $\ell_\mathrm{T}(g_\mathrm{T})=g_\mathrm{L}$. For every fixed element $t\in \mathrm{T}\setminus \set{g_\mathrm{T}}$, the restriction of $\ell_T$ to the set of predecessors of $t$ defines an order isomorphism onto an initial segment of $\mathrm{L}\setminus\set{g_T}$ (in particular, $\ell_\mathrm{T}\circ \operatorname{pred}_\mathrm{T}=  \operatorname{pred}_\mathrm{L}\circ \ell_\mathrm{T}$). Moreover, for any $t$ from $\mathrm{T}\setminus\{g_\mathrm{T}\}$ and any $y$ in $\mathrm{L}$ with $\ell_\mathrm{T}(t)\leq_\mathrm{L} y$ there is $t'$ in $T$ with $t\leq_\mathrm{T}t'$ such that $\ell_\mathrm{T}(t')=y$.
  \item $g_\mathrm{L}$ is not in the image of $\ell_\mathrm{A}\from \mathrm{A}\to \mathrm{L}$, nor in that of $\ell_\mathrm{B}\from \mathrm{B}\to \mathrm{L}$.
  \item For every $c\in\mathrm{L}\setminus\set{g_\mathrm{L}}$ the fibers $\ell_\mathrm{A}^{-1}(c)$ and $\ell_\mathrm{B}^{-1}(c)$ are infinite.
  \item $f(a,b)= g_\mathrm{T}$ if and only if  $\ell_\mathrm{A}(a)\ne \ell_\mathrm{B}(b)$. 
  \item If $\ell_\mathrm{A}(a)= \ell_\mathrm{B}(b)$, then $\ell_\mathrm{T}(f(a,b))=\ell_\mathrm{A}(a)$.
  \item At any level, $f$ defines a generic surjection: for any $c\in\mathrm{L}\setminus\{ g_\mathrm{L}\}$, any $t_1,\ldots, t_n$ from $\mathrm{T}\setminus\set{g_\mathrm{T}}$ and any pairwise distinct $a_1,\ldots,a_n$ from $\mathrm{A}$ such that $\ell_\mathrm{T}(t_i)=c=\ell_\mathrm{A}(a_i)$ for all $i$, there are infinitely many $b$ from $\mathrm{B}$ such that, for all $i$, we have  $f(a_i,b)=t_i$; similarly with the roles of $\mathrm{A}$ and $\mathrm{B}$ interchanged.
  \end{enumerate}
\end{defin}
\begin{pr}\label{Pr:QE} The following properties hold.
\begin{enumerate}
\item  $T$ is complete and admits quantifier elimination. 
\item The union of the definable sets $\mathrm{T}\setminus\{g_T\}$ and $ \mathrm{L}\setminus\{g_L\}$ is stably embedded, with induced structure a pure levelled (binary) meet-tree. In particular,  $\mathrm{L}\setminus\{g_L\}$ is 
stably embedded with induced structure a pure ordered set.
\end{enumerate}
\end{pr}

\begin{proof}
It is easy to see that $T$ is consistent. It is enough to prove quantifier elimination: since $T$ admits a prime $L$-substructure $M_0$, with underlying set the interpretations of the closed $L$-terms over $\emptyset$, completeness of $T$ will follow, and (2) is a direct consequence of quantifier elimination.

Let $N_0$ and $N_1$ be models of $T$ with $N_0$ countable and $N_1$ $\aleph_1$-saturated, and let $M$ be a common $L$-substructure of $N_0$ and $N_1$. It is an easy exercise to $M$-embed $N_0$ into $N_1$, yielding quantifier elimination.
\end{proof}

The theory induced on $\mathrm{T}\setminus\{g_T\}$ and $ \mathrm{L}\setminus\{g_L\}$ is thus precisely the one used in \cite[Fact~2.4.15]{bp}.

\begin{lemma}\label{lemma:linmax}
For all $M\models T$,  all linearly ordered definable subsets of $\operatorname{T}(M)$ have a maximum.
\end{lemma}
\begin{proof}
This follows from quantifier elimination. Alternatively, one may use that no infinite branch is definable in the standard binary meet-tree $(2^{<\omega},\omega)$, e.g., since for 
any $n\in \omega$ and branches $s,s'\in 2^{\omega}$ with $s\upharpoonright_n=s'\upharpoonright_n$ there is an automorphism $\sigma$ over $2^{< n}$ with $\sigma(s)=s'$.
\end{proof}

\section{The types}
Let $T$ be the theory defined in the previous section, and $\monster\models T$ a monster model. Failure of amalgamation of definable types boils down to the following phenomenon.  All elements $y$ of $\mathrm{A}$ with level larger than $\operatorname{L}(\monster)$ have the same, definable, type, and similarly for $z$ in $\mathrm{B}$; nevertheless, if such $y$ and $z$ have the \emph{same} infinite level, then $f(y,z)$ can be used to produce an externally definable subset of $\operatorname{T}(\monster)$ which is not definable. More formally, we proceed as follows.
\begin{defin}
 Define the following sets of $L(\monster)$-formulas.
  \begin{enumerate}
  \item $p(x)$ is the global type of an element $x$ of sort $\mathrm{L}$ such that $x>\operatorname{L}(\monster)$.
  \item $q_\mathrm{A}(x,y)$ restricts to $p$ on $x$, and says that $y$ is an element of sort $\mathrm{A}$ with $\ell_\mathrm{A}(y)=x$.
  \item $q_\mathrm{B}(x,z)$ restricts to $p$ on $x$, and says that $z$ is an element of sort $\mathrm{B}$ with $\ell_\mathrm{B}(y)=x$.
  \end{enumerate}
\end{defin}

    \begin{lemma} All of $p,q_\mathrm{A}, q_\mathrm{B}$ are complete types over $\monster$ which are $\emptyset$-definable.
  \end{lemma}
  \begin{proof}
    Consistency and $\emptyset$-definability are clear. As for completeness, we argue as follows.
      \begin{enumerate}
  \item Completeness of $p(x)$ follows from Proposition~\ref{Pr:QE}(2).
  \item As for $q_\mathrm{A}(x,y)$, note that, since $y$ has a new level, it cannot be in $\mathrm{A}(\monster)$.  Again because $\ell_\mathrm{A}(y)$ is new,  for all $b\in \mathrm{B}(\monster)$ we must have $f(y,b)=g_\mathrm{T}$. By quantifier elimination, this is enough to determine a complete type.
  \item The argument for $q_\mathrm{B}(x,z)$ is symmetrical.\qedhere
  \end{enumerate}
  \end{proof}
  \begin{pr}
The types $q_\mathrm{A}(x,y)$ and $q_\mathrm{B}(x,z)$ cannot be amalgamated over $p(x)$ into a definable type. In other words, no completion of $q_\mathrm{A}(x,y)\cup q_\mathrm{B}(x,z)$ is definable. 
  \end{pr}
  \begin{proof}
    Suppose $r(x,y,z)$ is a completion of $q_\mathrm{A}(x,y)\cup q_\mathrm{B}(x,z)$. Then $r(x,y,z)\proves \ell_\mathrm{A}(y)=x=\ell_\mathrm{B}(z)$,  thus $r(x,y,z)\proves \ell_\mathrm{T}(f(y,z))=x$. Since $p(x)$ is not realised, $f(y,z)$, having a new level, cannot be in $\operatorname{T}(\monster)$. Consider the set $\set{d\in \operatorname{T}(\monster)\mid r(x,y,z)\proves f(y,z)>d}$. If $r(x,y,z)$ is definable, then this set is definable. As a set of predecessors,  it must be linearly ordered, hence  have a maximum by \Cref{lemma:linmax}. But then $f(y,z)\notin \monster$ contradicts binary ramification.
  \end{proof}    
    \begin{co}
    In $T$, global definable types  do not have the amalgamation property.
    \end{co}

  This partially answers \cite[Question~9.3.1]{bp}, asking whether there is such a theory which, additionally, has uniform definability of types; note that $T$ does not. Moreover, $T$ is easily shown to have $\mathsf{IP}$.
  \begin{question}
    Is there a $\mathsf{NIP}$ theory where global definable types do not have the amalgamation property?
  \end{question}


\begin{thebibliography}{10}

\bibitem{BaPo98}
Yerzhan Baisalov and Bruno Poizat.
\newblock Paires de structures o-minimales.
\newblock {\em J. Symbolic Logic}, 63(2):570--578, 1998.

\bibitem{CD}
Pablo Cubides~Kovacsics and Fran\c{c}oise Delon.
\newblock Definable types in algebraically closed valued fields.
\newblock {\em MLQ Math. Log. Q.}, 62(1-2):35--45, 2016.

\bibitem{bp}
Pablo Cubides~Kovacsics, Martin Hils, and Jinhe Ye.
\newblock Beautiful pairs.
\newblock \url{https://arxiv.org/abs/2112.00651}, 2021.

\bibitem{HKR2016}
Martin Hils, Moshe Kamensky, and Silvain Rideau.
\newblock Imaginaries in separably closed valued fields.
\newblock {\em Proc. Lond. Math. Soc. (3)}, 116(6):1457--1488, 2018.

\bibitem{HiRi}
Martin Hils and Silvain Rideau-Kikuchi.
\newblock Un principe d'{Ax-Kochen-Ershov} imaginaire.
\newblock \url{https://arxiv.org/abs/2109.12189}, 2021.

\bibitem{HL}
Ehud Hrushovski and Fran\c{c}ois Loeser.
\newblock {\em Non-archimedean tame topology and stably dominated types},
  volume 192 of {\em Annals of Mathematics Studies}.
\newblock Princeton University Press, Princeton, NJ, 2016.

\bibitem{Joh20}
Will Johnson.
\newblock On the proof of elimination of imaginaries in algebraically closed
  valued fields.
\newblock {\em Notre Dame J. Form. Log.}, 61(3):363--381, 2020.

\bibitem{MS}
David Marker and Charles~I. Steinhorn.
\newblock Definable types in o-minimal theories.
\newblock {\em J. Symbolic Logic}, 59(1):185--198, 1994.

\bibitem{poizat}
Bruno Poizat.
\newblock Paires de structures stables.
\newblock {\em J. Symbolic Logic}, 48(2):239--249, 1983.

\bibitem{RidVDF}
Silvain Rideau.
\newblock Imaginaries and invariant types in existentially closed valued
  differential fields.
\newblock {\em J. Reine Angew. Math.}, 2016.

\bibitem{touchard}
Pierre Touchard.
\newblock Stably embedded submodels of {Henselian} valued fields.
\newblock \url{https://arxiv.org/abs/2005.02363}, 2020.

\end{thebibliography}
\end{document}